\documentclass[a4paper,twoside,12pt]{article}
\usepackage[english]{babel}

\usepackage{amsmath,amsfonts,amssymb,amsthm}
\newtheorem{teo}{Theorem}
\newtheorem{lem}{Lemma}
\newtheorem{deff}{Definition}
\newtheorem{rem}{Remark}

 \title{{\bf Asymptotics   of Eigenvalues for  Differential Operators of   Fractional Order}}

\author{Maksim \,V.~Kukushkin   \\ \\
 \small  \textit{Moscow State University of Civil Engineering}\\\textit{\small\textit{Russia, Moscow, 129337,}}\\
 \small  \textit{Kabardino-Balkarian Scientific Center, RAS, }\\\textit{\small\textit{Russia, Nalchik, 360051, kukushkinmv@rambler.ru}} }
\date{}
\begin{document}

\maketitle

\begin{abstract}
In this paper we  deal with a second order  multidimensional fractional differential operator. We consider a case where the  leading  term represented by the uniformly elliptic operator  and the final term is the Kipriyanov  operator of  fractional differentiation. We conduct  classification of such a type of  operators by   belonging of  their   resolvent to the Schatten-von Neumann  class and formulate  the  sufficient condition for  completeness of the root functions system. Finally we obtain an   asymptotic formula.
\end{abstract}
\begin{small}\textbf{Keywords:} Operators of fractional differentiation;  Weyl  asymptotic; Schatten-von Neumann  class;
  sectorial property; accretive property; operators with  a compact resolvent; root vectors.
\end{small}

\section{Introduction and    preliminary remarks}

Everywhere further in abstract  definitions  we consider   linear    operators acting in a separable complex  Hilbert space $\mathfrak{H}.$
Denote  by  $ C,C_{i} ,\,i\in \mathbb{N}_{0}$ arbitrary  positive constants.
Let $ \mathcal{B} (\mathfrak{H})$ be  a set of linear bounded operators acting in the Hilbert space $\mathfrak{H}.$     By  $\mathrm{D} (A),\, \mathrm{R} (A),\,\mathrm{N}(A)$ we   denote  the domain of definition,  range, and
inverse image of zero of the operator $A$ respectively.  We use the denotation $ \mathrm{nul}\, A:=\mathrm{dim}\, \mathrm{N}(A).$
  Denote by $R_{A}(\lambda),\,R_{A}(0):=R_{A}$    the resolvent of the operator $A$ and let  $\mathrm{P}(A)$ be    the   resolvent set of the operator $A.$ By $\tilde{A}$ we denote   the closure of the operator $A.$
  Using the definition \cite[p.46]{firstab_lit:1Gohberg1965}  we define  $s$-numbers for a compact operator $A.$ By definition, put
$
s_{i}(A)=\lambda_{i}(T),\,i=1,2,...,r(T),
$
where $  T=(A^{\ast}A)^{1/2},\,r(T)={\rm dim}\, \mathrm{R}(T).$
  Suppose $s_{i}=0,\,i=r(T)+1,...,$ if $r(T)<\infty.$  According  to the terminology of the monograph   \cite{firstab_lit:1Gohberg1965}  a dimension  of the  root vectors subspace  corresponding  to a certain  eigenvalue $\lambda_{k}$  is called {\it algebraic multiplicity} of the eigenvalue $\lambda_{k}.$
 Denote
by $\nu(A)$ the sum of all algebraic multiplicities of the operator $A.$   Let  $\mathfrak{S}_{p}(\mathfrak{H}),\, 0< p<\infty $ be       the Schatten-von Neumann    class and      $\mathfrak{S}_{\infty}(\mathfrak{H})$ be the set of compact operators. By definition, put
$$
\mathfrak{S}_{p}(\mathfrak{H}):=\left\{ T: \mathfrak{H}\rightarrow \mathfrak{H},  \sum\limits_{i=1}^{\infty}s^{p}_{i}(T)<\infty,\;0< p<\infty \right\}.
$$
  Let  $A_{\mathfrak{R}}:= \left(A+A^{*}\right)/2,\,A_{\mathfrak{I}}:= \left(A-A^{*}\right)/2i$ be
 a so-called real   and imaginary component  of the operator $A$  respectively.
 Denote by  $\mu(A) $   order of the densely defined  operator $A$ with a compact resolvent   if we have
the estimate  $s_{n}(R_{A})\leq   C \,n^{-\mu},\,n\in \mathbb{N},\,0\leq\mu<\infty.$ Let $\{x_{n}\}_{1}^{\infty},\{y_{n}\}_{1}^{\infty}$ be sequences consist of positive real numbers. If there exist constants  $c_{1},c_{2}>0$ such that   $ c_{1} x_{n}\leq y_{n}\leq c_{2} x_{n},\,n\in \mathbb{N},$  then we write $x_{n}\asymp y_{n}.$
In accordance with  the terminology of the monograph  \cite{firstab_lit:kato1980} the set $\Theta(A):=\{z\in \mathbb{C}: z=(Af,f)_{\mathfrak{H}},\,f\in \mathrm{D}(A),\,\|f\|_{\mathfrak{H}}=1\}$ is called a   {\it numerical range}  of the  operator $A.$
Denote by  $\widetilde{\Theta}(A) $ the  closure of the  set $\Theta(A).$
 We use the definition of  the   sectorial property given in \cite[p.280]{firstab_lit:kato1980}. An  operator $A$ is called   a {\it sectorial } operator  if its  numerical range   belongs to a  closed
sector     $\mathfrak{ L}_{\gamma}(\theta):=\{\zeta:\,|\arg(\zeta-\gamma)|\leq\theta<\pi/2\} ,$ where      $\gamma$ is a vertex   and  $ \theta$ is a semi-angle of the sector   $\mathfrak{ L}_{\gamma}(\theta).$
We  shall  say that the  operator $A$ has a positive sector       if $\mathrm{Im}\,\gamma=0,\,\gamma>0.$
According to the  terminology  of the monograph   \cite{firstab_lit:kato1980} an operator $A$ is called    {\it strictly  accretive}   if the following relation  holds  $\mathrm{Re}(Af,f)_{\mathfrak{H}}\geq C\|f\|^{2}_{\mathfrak{H}},\,f\in \mathrm{D}(A).$ In accordance with  the definition  \cite[p.279]{firstab_lit:kato1980}   an  operator $A$ is called    {\it m-accretive}     if the next relation  holds $(A+\zeta)^{-1}\in \mathcal{B}(\mathfrak{H}),\,\|(A+\zeta)^{-1}\| \leq   (\mathrm{Re}\zeta)^{-1},\,\mathrm{Re}\zeta>0. $
An operator $A$ is called    {\it m-sectorial}   if $A$ is   sectorial    and $A+ \beta$ is m-accretive   for some constant $\beta.$   An operator $A$ is called     {\it symmetric}     if one is densely defined and the next equality  holds $(Af,g)_{\mathfrak{H}}=(f,Ag)_{\mathfrak{H}},\,f,g\in  \mathrm{D} (A).$
A symmetric operator is called     {\it positive}        if       values of its  quadratic form  are nonnegative. Denote by $\mathfrak{H}_{A},\,\|\cdot\|_{A}$ the  energetic space generated by the operator $A$  and the norm on this space respectively   (see  \cite{firstab_lit:Eb. Zeidler},\cite{firstab_lit:mihlin1970}).
In accordance with  the denotation of the paper   \cite{firstab_lit:kato1980}   we consider a      sesquilinear form   $ \mathfrak{t}  [\cdot,\cdot]$
defined on a linear manifold  of the Hilbert space $\mathfrak{H}$ (further we use the term {\it form}).  Denote by $  \mathfrak{t} [\cdot ]$ the quadratic form corresponding to the sesquilinear form $\mathfrak{t}  [\cdot,\cdot].$
 Denote   by  $\mathfrak{Re} \,  \mathfrak{t}=(\mathfrak{t}+\mathfrak{t}^{\ast})/2,\,\mathfrak{Im}\,\mathfrak{t}   =(\mathfrak{t}-\mathfrak{t}^{\ast})/2i  $
  the real and imaginary component    of the   form $ \mathfrak{t}$ respectively, where $\mathfrak{t}^{\ast}[u,v]=\mathfrak{t} \overline{[v,u]},\;\mathrm{D}(\mathfrak{t} ^{\ast})=\mathrm{D}(\mathfrak{t}).$ According to these definitions, we have $\mathfrak{Re} \,
 \mathfrak{t}[\cdot]=\mathrm{Re}\,\mathfrak{t}[\cdot],\,\mathfrak{Im} \, \mathfrak{t}[\cdot]=\mathrm{Im}\,\mathfrak{t}[\cdot].$ Denote by $\tilde{\mathfrak{t}}$ the closure   of the   form $\mathfrak{t}\,.$ The range of the quadratic form   $\mathfrak{t}[f],\,f\in \mathrm{D}(\mathfrak{t}),\,\|f\|_{\mathfrak{H}}=1$ is called a  {\it range} of the sesquilinear form  $\mathfrak{t} $ and is denoted by $\Theta(\mathfrak{t}).$
 A  form $\mathfrak{t}$ is called    {\it sectorial}    if  its    range  belongs to a sector  having  a  vertex $\gamma$  situated at the real axis and a semi-angle $0\leq\theta<\pi/2.$  Suppose   $\mathfrak{l}$ is a closed sectorial form; then  a linear  manifold  $\mathrm{D}'\subset\mathrm{D} (\mathfrak{l})$   is
called a core of $\mathfrak{l}$ if the restriction   of $\mathfrak{l}$ to   $\mathrm{D}'$ has the   closure
$\mathfrak{l}.$    Due to Theorem 2.7 \cite[p.323]{firstab_lit:kato1980}  there exist unique    m-sectorial operators  $A_{\mathfrak{l}},A_{\mathfrak{Re}\, \mathfrak{l}} $  associated  with   the  closed sectorial   forms $\mathfrak{l},\mathfrak{Re}\, \mathfrak{l}$ respectively.  The operator  $A_{\mathfrak{Re}\, \mathfrak{l}} $ is called a {\it real part} of the operator $A_{\mathfrak{l}}$ and is denoted by  $Re\, A_{\mathfrak{l}}.$ Suppose  $A$ is a sectorial densely defined operator and $\mathfrak{k}[u,v]:=(Au,v)_{\mathfrak{H}},\,\mathrm{D}(\mathfrak{k})=\mathrm{D}(A);$  then
 due to   Theorem 1.27 \cite[p.318]{firstab_lit:kato1980}   the   form $\mathfrak{k}$ is   closable, due to
   Theorem 2.7 \cite[p.323]{firstab_lit:kato1980} there exists the unique m-sectorial operator   $T_{\tilde{\mathfrak{k}}}$ associated  with  the form $\tilde{\mathfrak{k}}.$  In accordance with the  definition \cite[p.325]{firstab_lit:kato1980} the    operator $T_{\tilde{\mathfrak{k}}}$ is called a  {\it Friedrichs extension} of the operator $A.$
Everywhere further,  unless otherwise stated,   we  use the notations of \cite{firstab_lit:kato1980}, \cite{firstab_lit:1Gohberg1965}, \cite{firstab_lit:kipriyanov1960}, \cite{firstab_lit:1kipriyanov1960}, \cite{firstab_lit:samko1987}.

In accordance with  the  notation  of the paper   \cite{firstab_lit:kipriyanov1960} we assume that $\Omega$ is a convex domain of the $n$ - dimensional Euclidean space $\mathbb{E}^{n}$, $P$ is a fixed point of the boundary $\partial\Omega,$ and
$Q(r,\vec{\mathbf{e}})$ is an arbitrary point of $\Omega.$   Denote by $\vec{\mathbf{e}}$ an  unit vector having  the  direction from $P$ to $Q,$   by $r=|P-Q|$ an Euclidean distance between the points $P$ and $Q.$
We   consider the  Lebesgue classes  $L_{p}(\Omega),\;1\leq p<\infty $ of complex valued functions.
By definition,  put
\begin{equation*}
\int\limits_{\Omega}|f(Q)|^{p}dQ=\int\limits_{\omega}d\chi\int\limits_{0}^{d(\vec{\mathbf{e}})}|f(Q)|^{p}r^{n-1}dr<\infty,\,f\in L_{p}(\Omega),
\end{equation*}
where $d\chi$ is an element of a solid angle of
the  the  unit sphere surface    in the space $\mathbb{E}^{n}$ and $\omega$   is the  surface of this sphere,   $d:=d(\vec{\mathbf{e}})$   is a length of the  segment of the  ray going from the point $P$ in the direction
$\vec{\mathbf{e}}$ within the domain $\Omega.$
Without lose of   generality,  we consider only those directions of $\vec{\mathbf{e}}$ for which the inner integral  exists and  is  finite. This is the well known fact that they are almost all directions.
We use  the shorthand notation $P\cdot Q=P^{i}Q_{i}=\sum^{n}_{i=1}P_{i}Q_{i}$ for the inner product of the points
$P=(P_{1},P_{2},...,P_{n}),\,Q=(Q_{1},Q_{2},...,Q_{n}),\, P,Q\in \mathbb{E}^{n}.$
      Denote by $D_{i}f$  a  generalized derivative of the function $f$ with respect to the coordinate variable     $x_{i},\,1\leq i\leq n.$
We   assume that all functions has a zero extension outside  of $\bar{\Omega}.$
Let   $W_2 ^l(\Omega)$ be   Sobolev  spaces or according to  another  denotation $H^{l}(\Omega):=W_2 ^l(\Omega),\,l=0,1,...\,.$         We consider the space $H^{1}_{0}(\Omega)$ endowed with   the norm
$$
 \|\cdot\|_{H_{0}^{1}(\Omega)}=  \left(\sum\limits_{i=1}^{n} \int\limits_{\Omega}   \left| (D_{i}\cdot) (Q)\right|^{2}dQ \right)^{1/2}.
$$
  We also use the shorthand denotation  $L_{p}:=L_{p}(\Omega),\, H^{1}_{0}:=H^{1}_{0}(\Omega).$
Denote by  ${\rm Lip}\, \mu,\;0<\mu\leq1 $      the set of  functions satisfying the Holder-Lipschitz condition
$$
{\rm Lip}\, \mu:=\left\{\rho(Q):\;|\rho(Q)-\rho(P)|\leq   \mathcal{M } \, r^{\mu},\;P,Q\in \bar{\Omega}\right\}.
$$
Consider   Kipriyanov's  fractional differential  operator      defined by the following  formal expression  (see \cite{firstab_lit:1kipriyanov1960})
\begin{equation*}
(\mathfrak{D}^{\alpha}f)(Q)=\frac{\alpha}{\Gamma(1-\alpha)}\int\limits_{0}^{r} \frac{[f(Q)-f(P+\vec{\mathbf{e}}t)]}{(r - t)^{\alpha+1}} \left(\frac{t}{r} \right) ^{n-1} dt+
C^{(\alpha)}_{n} f(Q) r ^{ -\alpha},\; P\in\partial\Omega,
\end{equation*}
where
$
C^{(\alpha)}_{n} = (n-1)!/\Gamma(n-\alpha).
$
 The properties of this operator were studied in the papers  \cite{firstab_lit:kipriyanov1960}-\cite{firstab_lit:2.2kipriyanov1960},\cite{firstab_lit:1.1kukushkin2017}.  It is easy  to see that  the  Kipriyanov   operator coincides with the Marchaud  operator  \cite{firstab_lit:samko1987}  on the segment in one dimensional case.
 According to     Theorem 2  \cite{firstab_lit:1kipriyanov1960} the Kipriyanov  operator   acts  as follows
\begin{equation}\label{1}
\mathfrak{D}^{\alpha}:H_{0}^{1}(\Omega)\rightarrow L_{2}(\Omega),
  \;0<\alpha<1,\,(n=2,3,...),
\end{equation}
 The case (n=1) is also  obtained  due to the property of the Marchaud operator (see \cite{firstab_lit:samko1987}).
In accordance with the definition given in  the paper  \cite{firstab_lit:1kukushkin2018} we consider the directional  fractional integrals.  By definition, put
$$
 (\mathfrak{I}^{\alpha}_{0+}f)(Q  ):=\frac{1}{\Gamma(\alpha)} \int\limits^{r}_{0}\frac{f (P+t\vec{\mathbf{e}})}{( r-t)^{1-\alpha}}\left(\frac{t}{r}\right)^{n-1}\!\!\!\!dt,\,(\mathfrak{I}^{\alpha}_{d-}f)(Q  ):=\frac{1}{\Gamma(\alpha)} \int\limits_{r}^{d }\frac{f (P+t\vec{\mathbf{e}})}{(t-r)^{1-\alpha}}\,dt,
$$
$$
\;f\in L_{p}(\Omega),\;1\leq p\leq\infty.
$$
Also,     we   consider auxiliary operators   the so-called   truncated directional  fractional derivatives    (see \cite{firstab_lit:1kukushkin2018}).  By definition, put
 \begin{equation*}
 ( \mathfrak{D} ^{\alpha}_{0+,\,\varepsilon}f)(Q)=\frac{\alpha}{\Gamma(1-\alpha)}\int\limits_{0}^{r-\varepsilon }\frac{ f (Q)r^{n-1}- f(P+\vec{\mathbf{e}}t)t^{n-1}}{(  r-t)^{\alpha +1}r^{n-1}}   dt+\frac{f(Q)}{\Gamma(1-\alpha)} r ^{-\alpha},\;\varepsilon\leq r\leq d ,
 $$
 $$
 (\mathfrak{D}^{\alpha}_{0+,\,\varepsilon}f)(Q)=  \frac{f(Q)}{\varepsilon^{\alpha}}  ,\; 0\leq r <\varepsilon ;
\end{equation*}
\begin{equation*}
 ( \mathfrak{D }^{\alpha}_{d-,\,\varepsilon}f)(Q)=\frac{\alpha}{\Gamma(1-\alpha)}\int\limits_{r+\varepsilon }^{d }\frac{ f (Q)- f(P+\vec{\mathbf{e}}t)}{( t-r)^{\alpha +1}} dt
 +\frac{f(Q)}{\Gamma(1-\alpha)}(d-r)^{-\alpha},\;0\leq r\leq d -\varepsilon,
 $$
 $$
  ( \mathfrak{D }^{\alpha}_{d-,\,\varepsilon}f)(Q)=      \frac{ f(Q)}{\alpha} \left(\frac{1}{\varepsilon^{\alpha}}-\frac{1}{(d -r)^{\alpha} }    \right),\; d -\varepsilon <r \leq d .
 \end{equation*}
  Now we can  consider the directional   fractional derivatives. By definition, put
 \begin{equation*}
 \mathfrak{D }^{\alpha}_{0+}f=\lim\limits_{\stackrel{\varepsilon\rightarrow 0}{ (L_{p}) }} \mathfrak{D }^{\alpha}_{0+,\varepsilon} f  ,\;
  \mathfrak{D }^{\alpha}_{d-}f=\lim\limits_{\stackrel{\varepsilon\rightarrow 0}{ (L_{p}) }} \mathfrak{D }^{\alpha}_{d-,\varepsilon} f ,\,1\leq p\leq\infty.
\end{equation*}
The properties of these operators  are  described  in detail in the paper  \cite{firstab_lit:1kukushkin2018}. Similarly to the monograph \cite{firstab_lit:samko1987} we consider   left-side  and  right-side cases. For instance, $\mathfrak{I}^{\alpha}_{0+}$ is called  a left-side directional  fractional integral  and $ \mathfrak{D }^{\alpha}_{d-}$ is called a right-side directional fractional derivative.
We   consider a  second order differential operator   with the Kipriyanov  fractional derivative  in  the final term
\begin{equation}\label{2}
 Lu:=-  D_{j} ( a^{ij} D_{i}f) \, +\rho\, \mathfrak{D}^{ \alpha } f ,\;\;
 $$
  $$
  \mathrm{D } (L)=H^{2}(\Omega)\cap H^{1}_{0}(\Omega),
 \end{equation}
we assume  that the  coefficients are  real-valued and satisfy the following conditions
 \begin{equation}\label{3}
 a^{ij}(Q)_{ i,j=\overline{1,n}  }\in C^{1}(\bar{\Omega}),\,\mathcal{A}= \sup\limits_{Q\in \Omega} \left(\sum\limits_{i,j=1}^{n}
|a_{ij}(Q)|^{2} \right)^{1/2}\!\!\!,\,  a^{ij}\xi _{i}  \xi _{j}  \geq   a  |\xi|^{2} ,\,  a  >0 ,
$$
$$
\rho(Q)\in {\rm Lip\,\mu},\, \mu>\alpha ,\,0<\varrho<\rho(Q)< \mathcal{P}.
\end{equation}
We also   consider a formal adjoint  operator
\begin{equation*}
 L^{+}f:=-  D_{i} ( a^{ij} D_{j}f)  + \mathfrak{D}  ^{\alpha}_{ d -}\rho f  ,\;\mathrm{D}(L^{+})=\mathrm{D}(L).
 \end{equation*}
For the sake of   simplicity, we should formulate the following trivial propositions   which are not  worth studying in the main part of the paper.
Using the results of the  the paper \cite{firstab_lit:1kukushkin2018} it can easily be checked that
\begin{equation}\label{4}
\mathrm{Re}(f,\rho\,\mathfrak{D}^{ \alpha } f)_{L_{2}(\Omega )}\geq\frac{1}{\eta^{2}}\|f\|^{2}_{L_{2}(\Omega )}, \,\eta\neq0 ,\,f\in H_{0}^{1}(\Omega),
\end{equation}
for $\varrho$  sufficiently large.  Using condition \eqref{3}, we obtain
\begin{equation}\label{5}
-\mathrm{Re}( D_{j} [ a^{ij} D_{i}f],f)_{L_{2} }\geq a\|f\|^{2}_{H^{1}_{0} }  ,\,f\in \mathrm{D}(L),
\end{equation}
 Combining \eqref{4} and \eqref{5}, we get
\begin{equation}\label{6}
\Theta(L),\Theta(L^{+})\subset \{ z\in    \mathbb{C}  : \mathrm{Re}z>C\},
\end{equation}
 where $\eta$ depends on parameters of expression \eqref{2}. Under these  assumptions, due to the results of the paper \cite{firstab_lit(JFCA)kukushkin2018}, we have
\begin{equation}\label{7}
\mathrm{R}(L)=\mathrm{R}(L^{+})=L_{2}(\Omega).
\end{equation}
Now it can easily be checked that $L^{+}=L^{\ast}.$ Hence by virtue of the well-known theorem we   conclude that $R^{\ast}_{L^{\,} }=R^{\,}_{L^{\ast}}.$
Combining \eqref{6}  and \eqref{7},
  we will have no difficulty in establishing the fact $\tilde{L}=L,$  see  Problem 5.15 \cite[p.165]{firstab_lit:kato1980}.
 Further,    we   use  the shorthand  notation   $H:=L_{\mathfrak{R}},\,  V:=\left(R_{L}\right)_{\mathfrak{R}}.$ Arguing as above, we see that  $H$ is closed and
  $\mathrm{R}(H)=L_{2}(\Omega).$ Due to the  reasoning of    Theorem 4.3 \cite{firstab_lit:1kukushkin2018},  we know  that $H$ is selfadjoint.  By virtue of \eqref{6}, we have
\begin{equation}\label{8}
\Theta(H) \subset \{ z\in    \mathbb{R}  :  z > C\}.
\end{equation}
Using  \eqref{8}, we obtain that $R_{H}$ is selfadjoint and positive, see Theorem 3 \cite[p.136]{firstab_lit: Ahiezer1966}.
Since $V$ is  symmetric   and $\mathrm{D}(V)=L_{2}(\Omega),$ then    $V$ is selfadjoint, see Theorem 1 \cite[p.136]{firstab_lit: Ahiezer1966}.  In the same way, we conclude that $V$ is positive. Note  that  due to Theorem 5.3 \cite{firstab_lit:1kukushkin2018} the operators $V$ and $R_{H}$ are compact. Further, if it is not stated overwise, we suppose  that inequality \eqref{4} is fulfilled.

Before our main consideration, we would like to make some remarks. Particularly we will discuss possible methods for achieving prospective results.
 The eigenvalue problem is still relevant for second order fractional  differential operators. Many papers were  devoted to this question, for instance the papers
\cite{firstab_lit:1Nakhushev1977}, \cite{firstab_lit:1Aleroev1984}-\cite{firstab_lit:3Aleroev2000}, \cite{firstab_lit:2kukushkin2017}.   The singular number   problem for a resolvent of  the second order differential operator
  with   Riemman-Liuvile's fractional derivative in the  final  term  is considered in the paper \cite{firstab_lit:1Aleroev1984}. It is proved
that the resolvent   belongs to the Hilbert-Shmidt class.  The problem of completeness
of  the root functions  system is researched in the paper \cite{firstab_lit:Aleroev1989}, also a similar problem is considered
in the paper \cite{firstab_lit:3Aleroev2000}. We would like to research a multidimensional case which can be reduced to
the cases considered in the papers listed above. For this purpose we deal with the extension of the
Kiprianov  fractional differential operator considered in detail in the papers \cite{firstab_lit:kipriyanov1960}-\cite{firstab_lit:2.2kipriyanov1960}.
We will apply  a  method of researching based on properties of the real component of the operator, but first let us verify an opportunity of using results obtained previously. For instance the studied   operator \eqref{2} can be represented by a sum and can be considered as a  perturbation of one operator by another one. In accordance with this   let us say a couple of words on the perturbation theory.

 It is remarkable that initially the perturbation theory of selfadjoint operators was born in  the  works of M. Keldysh \cite{firstab_lit:1keldysh 1951}-\cite{firstab_lit:3keldysh 1971} and had been  motivated by  the   works of such famous scientists as T. Carleman \cite{firstab_lit:1Carleman} and Ya.  Tamarkin  \cite{firstab_lit:1Tamarkin}. Over time, many papers were  published within  the  framework of this theory, for instance    F. Browder \cite{firstab_lit:1Browder},  M. Livshits \cite{firstab_lit:1Livshits}, B. Mukminov \cite{firstab_lit:1Mukminov}, I. Glazman \cite{firstab_lit:1Glazman}, M. Krein \cite{firstab_lit:1Krein}, B. Lidsky \cite{firstab_lit:1Lidskii},  A. Marcus \cite{firstab_lit:1Markus},\cite{firstab_lit:2Markus}, V. Matsaev \cite{firstab_lit:1Matsaev}-\cite{firstab_lit:3Matsaev}, S. Agmon \cite{firstab_lit:2Agmon}, V. Katznelson \cite{firstab_lit:1Katsnel'son}.
Nowadays'  there  exists   a  huge amount of theoretical results formulated  in the paper  of A. Shkalikov    \cite{firstab_lit:Shkalikov A.}. However for using these results we must have  a representation   of an initial operator by  the sum of   a  main part  a  so-called   non-perturbing operator and an  operator-perturbation.
  It is essentially the main part must be an operator of a special type either  selfadjoint or normal operator. If we consider a case  where  in the representation   the  main part  neither selfadjoint nor normal and we can not approach the  required representation in an  obvious  way, then we can use  the  other technique  based on    properties  of the real component of the initial operator.

We stress that  both summands of sum \eqref{2} are
neither selfadjoint  no normal   operators. We also   notice  that  search for  a convenient representation  of $L$ by a sum
of a selfadjoint operator and an operator-perturbation does not seem to be a reasonable   way. Now  to justify this  claim  we consider one of  possible representations of $L$ by  a sum.   Consider the
  Hilbert space $\mathfrak{H}$ and a selfadjoint strictly accretive operator
$  \mathcal{T}  :\mathfrak{H}\rightarrow \mathfrak{H}. $

\begin{deff}
In accordance with the   definition of the paper  \cite{firstab_lit:Shkalikov A.},  a quadratic form $\mathfrak{a}: = \mathfrak{a}[f]$ is called a $\mathcal{T}$ - subordinated form if the following condition holds
\begin{equation}\label{9}
\left|\mathfrak{a}[f]\right| \leq b\, \mathfrak{t}[f] +M\|f\|^{2}_{\mathfrak{H}},\;
   \mathrm{D} (\mathfrak{a}) \supset  \mathrm{D}(\mathfrak{t}),\; b < 1,\; M > 0,
\end{equation}
where $\mathfrak{t}[f]=\|\mathcal{T}^{\frac{1}{2}}\|^{2}_{\mathfrak{H}},\,f\in \mathrm{D}(\mathcal{T}^{\frac{1}{2}}).$
 The form $\mathfrak{a}$  is called  a completely $\mathcal{T}$ - subordinated form   if besides  of \eqref{9}  we have the following additional condition  $\forall \varepsilon>0 \,\exists b,M>0:\, b <\varepsilon.$
\end{deff}
Let us consider a trivial  decomposition  of the operator $L$ on the sum $L=2 L_{\mathfrak{R}}-L^{+}$ and let us use the notation $\mathcal{T}:=2L_{\mathfrak{R}},\, \mathcal{A}:=-L^{+}.$ Then we have $L=\mathcal{T}+\mathcal{A}.$
Due to the sectorial property
proven in   Theorem 4.2 \cite{firstab_lit:1kukushkin2018}  we have
\begin{equation}\label{10}
 |(\mathcal{A}f,f)_{L_{2} }| \!=\!\sec \theta_{f}\, |\mathrm{Re}(\mathcal{A}f,f)_{L_{2} }|\! = \!\sec \theta_{f}\, \frac{1}{2}(\mathcal{T}f,f)_{L_{2} },\,f\in \mathrm{D}(\mathcal{T}),
\end{equation}
 where $ 0\leq\theta_{f}\leq \theta,\, \theta_{f}:=\left|\mathrm{arg}  (L^{+}f,f)_{L_{2} }\right|,\, L_{2}:=L_{2}(\Omega)  $   and $\theta$ is a semi-angle   corresponding to the sector $\mathfrak{L}_{0}.$
  Due to Theorem 4.3 \cite{firstab_lit:1kukushkin2018}  the operator $\mathcal{T}$ is m-accretive. Hence in consequence of Theorem
3.35  \cite[p.281]{firstab_lit:kato1980} $\mathrm{D}(\mathcal{T} )$ is a core of the operator $\mathcal{T}^{\frac{1}{2}}.$ It implies that we can extend   relation \eqref{10} to
\begin{equation}\label{11}
 \frac{1}{2} \, \mathbf{\mathfrak{t}}[f]\leq|\mathbf{\mathfrak{a }}[f]|\leq \sec \theta \frac{1}{2} \, \mathbf{\mathfrak{t}}[f],\, f\in \mathrm{D} (\mathbf{\mathfrak{t}}),
\end{equation}
where $\mathfrak{a}$ is a quadratic form generated by $\mathcal{A}$ and $\mathfrak{t}[f] = \|\mathcal{T}^{\frac{1}{2}}f\|^{2}_{\mathfrak{H}}.$
 If we consider
the case $0<\theta<\pi/3,$ then it is obvious  that there exist  constants  $b < 1$ and $M > 0$ such that
the following inequality holds
$$
|\mathfrak{a}[f]| \leq b\, \mathfrak{t}[f] +M\|f\|^{2}_{L_{2}},\;
 f\in \mathrm{D}(\mathfrak{t}).
$$
Hence the form $\mathfrak{a}$ is a $\mathcal{T}$ - subordinated form.
In accordance with  the definition given in the paper \cite{firstab_lit:Shkalikov A.} it means $\mathcal{T}$ - subordination of the operator $\mathcal{A}$ in the sense of   form. Assume
that $\forall \varepsilon>0\,\exists b,M>0:\,  b <\varepsilon.$  Using  inequality \eqref{11},  we get
$$
   \frac{1}{2}\mathfrak{t}[f]   \leq \varepsilon \, \mathfrak{t}[f] + M(\varepsilon)\|f\|^{2}_{L_{2} };\; \mathfrak{t}[f]\leq \frac{2 M(\varepsilon)}{(1-2\varepsilon)}\|f\|^{2}_{L_{2} },  \,f\in  \mathrm{D}(\mathfrak{t}),\, \varepsilon < 1/2.
$$
Using the strictly accretive property of the operator $L$ (see inequality  (4.9) \cite{firstab_lit:1kukushkin2018}), we   obtain
$$
\|f\|^{2}_{H^{1}_{0} } \leq C  \mathfrak{t}[f],  \,f\in  \mathrm{D}(\mathfrak{t}).
$$
On the other hand, using results of the paper  \cite{firstab_lit:1kukushkin2018}  it is easy to prove that $H^{1}_{0}(\Omega)\subset\mathrm{D}(\mathfrak{t}).$
Taking into account the facts considered above,   we get
$$
\|f\|_{H^{1}_{0} }   \leq  C \|f\|_{L_{2}},  \,f\in H^{1}_{0}(\Omega)  .
$$
    It can not be! It is well-known fact. This contradition shows us that the form $\mathfrak{a}$ is not   the completely $\mathcal{T}$ - subordinated form. It implies that we can not use     Theorem 8.4 \cite{firstab_lit:Shkalikov A.}.
Note that a reasoning corresponding to  another  trivial representation  of $L$ by  a sum  is analogous.
This rather particular example does not aim to show the inability of using   remarkable
methods considered in the paper  \cite{firstab_lit:Shkalikov A.}  but only creates   prerequisite for  some valuableness of  another
method  based on  using spectral properties of the real component of the  initial operator $L.$ Now we would like to demonstrate effectiveness   of   this method.

\section{ Auxiliary lemmas}

 It was proved in the paper   \cite{firstab_lit:1kukushkin2018}  that the operator $L$ is sectorial. Note that  we used   Theorem 2 \cite{firstab_lit:1kipriyanov1960} to obtain   coordinates  of the sector vertex  $\gamma$ and   values of the sector   semi-angle  $\theta,$ but this theorem is   useful in the general case where  a considered space is the Nicodemus space.
  Note that  we can    improve  estimate (5) \cite{firstab_lit:1kipriyanov1960} in the case corresponding to the values of indexes  $p=2,\,l=1,$ thus  making the estimate   more convenient for finding     values of $\gamma$ and    $\theta.$

\begin{lem}\label{L1} We have the following estimate
 \begin{equation}\label{12}
 \|  \mathfrak{D }^{\alpha}_{0+ }  f \|_{L_{2}}\leq   \mathcal{ K}  \|f\|_{H_{0}^{1}},\,f\in H_{0}^{1}(\Omega),\,\mathcal{ K}>0 .
\end{equation}
\end{lem}
\begin{proof}
 First  assume that   $f\in C_{0}^{\infty}(\Omega),$ then in accordance with Lemma 2.5   \cite{firstab_lit:1kukushkin2018}, we have $\mathfrak{D }^{\alpha}_{0+ }  f =\mathfrak{D }^{\alpha}   f.$ Applying the  triangle inequality, we get
$$
\| \mathfrak{D }^{\alpha}  f \|_{L_{2} }\leq\frac{\alpha}{\Gamma(1-\alpha)}\left(\int\limits_{\Omega} \left|\int\limits_{0 }^{r}\frac{   f (Q)-
 f (P+\vec{\mathbf{e}}t )}{ ( r-t)^{\alpha +1}} dt\right|^{2} dQ    \right)^{\frac{1}{2}}
$$
$$
 +C_{n}^{(\alpha)}\left(\int\limits_{\Omega}\left|  f(Q)r^{-\alpha}   \right|^{2}dQ    \right)^{\frac{1}{2}}=\frac{\alpha}{\Gamma(1-\alpha)} I_1 +C_{n}^{(\alpha)}I_2 .
$$
Making the change of  variables,  applying the  generalized Minkowskii  inequality, we get
$$
I_1  =
    \left(\int\limits_{\Omega}\left|\int\limits_{0}^{ r}\frac{  \left[f (Q)-  f (Q-\vec{\mathbf{e}}t)\right] }{ t^{\alpha +1}} dt\right|^{2} dQ    \right)^{\frac{1}{2}}
\leq   \int\limits_{ 0}^{\delta}t^{-\alpha -1}\!\!\left(\int\limits_{\Omega} \left| f(Q)-f(Q-\vec{\mathbf{e}}t) \right|^{2}  dQ   \right)^{\frac{1}{2}} dt,
 $$
where $\delta={\rm diam}\, \Omega.$
Let us rewrite the previous inequality as follows
$$
I_{1}\leq   \int\limits_{ 0}^{ \delta}t^{-\alpha -1}\!\!\left(\int\limits_{\Omega} \left| \int\limits_{0}^{t} f' (Q-\vec{\mathbf{e}}\tau) d\tau\right|^{2}   dQ\right)^{\frac{1}{2}}\!\!dt,\;
 f' (Q ):=\lim\limits_{t\rightarrow 0}\frac{ f(Q-\vec{\mathbf{e}}t)-f(Q)}{t}\,.
$$
Using  the Cauchy-Schwarz inequality with respect to the inner integral, applying  the  Fubini   theorem, taking  into account the fact   that the function $f$   vanishes  outside of    $\Omega,$ we obtain
$$
I_{ 1} \leq     \int\limits_{ 0}^{ \delta}t^{-\alpha -1}\!\!\left(\int\limits_{\Omega}dQ \int\limits_{0}^{t} \left|f'  (Q-\vec{\mathbf{e}}\tau)\right|^{2}d\tau \int\limits_{0}^{t}d\tau    \right)^{\frac{1}{2}} dt
$$
$$
 =   \int\limits_{ 0}^{ \delta}t^{-\alpha -1/2 }\left( \int\limits_{0}^{t}d\tau \int\limits_{\Omega}\left|f'  (Q-\vec{\mathbf{e}}\tau)\right|^{2} dQ   \right)^{\frac{1}{2}} dt\leq  \frac{\delta ^{1-\alpha}  }{ 1-\alpha  }  \,  \|f' \|_{L_{2} }.
$$
Arguing as above, we get
$$
   I_2  =
  \left(\int\limits_{\Omega}   r  ^{-2\alpha  } \left|\int\limits_{0}^{r} f' (Q-\vec{\mathbf{e}}t)\,dt\right|^{2}\!\!\!dQ     \right)^{\frac{1}{2}}\leq \left\{\int\limits_{\Omega}     \left(\int\limits_{0}^{ r} | f' (Q-\vec{\mathbf{e}}t) | r  ^{ -\alpha  }\,dt\right)^{2}\!\!\!dQ     \right\}^{\frac{1}{2}}
$$

$$
 \leq  \left\{\int\limits_{\Omega}    \left(\int\limits_{0}^{r} |f' (Q-\vec{\mathbf{e}}t)|t^{-\alpha}dt\right)^{2}\!\!\!dQ     \right\}^{\frac{1}{2}}\leq
  \left\{\int\limits_{\Omega}    \left(\int\limits_{0}^{\delta} |f' (Q-\vec{\mathbf{e}}t)|t^{-\alpha}dt\right)^{2}\!\!\!dQ     \right\}^{\frac{1}{2}} .
 $$
 Applying  the  generalized Minkowski inequality,  we get
$$
I_{2}\leq
\int\limits_{0 }^{ \delta}\!\! t^{- \alpha}  \left(\int\limits_{\Omega} \left| f'  (Q+\vec{\mathbf{e}}t)\right|^{2}  dQ      \right)^{\!\!\frac{1}{2}}\!\!\!dt
\leq \frac{  \delta ^{1-\alpha}}{    1-\alpha  \,}\, \| f'\|_{L_{2}}.
$$
Finally, we obtain
\begin{equation}\label{13}
\| \mathfrak{D}^{\alpha}_{0+}  f \|_{L_{2}}\leq C\| f'\|_{L_{2}},   \;f\in C_{0}^{\infty}(\Omega).
\end{equation}
  Since $|f'(Q)|=|\nabla f \cdot \vec{\mathbf{e}}|\leq  |\nabla f |,$ then  it is easy to prove that
$
\| f'\|_{L^{  }_{2}}\leq   \|f\|_{H^{ 1 }_{0}},
   f\in C_{0}^{\infty}(\Omega).
$
Hence
\begin{equation}\label{14}
 \| \mathfrak{D}^{\alpha}_{0+}  f \|_{L_{2}}\leq C\|f\|_{H^{1}_{0}},   \;f\in C_{0}^{\infty}(\Omega).
\end{equation}
Now assume that  $f\in  H_{0}^{1}(\Omega),$ then there  exists the sequence $\{f_{n}\}\in C_{0}^{\infty}(\Omega): f_{n}\stackrel{H^{1}_{0}}{\longrightarrow}f.$ Applying   inequality \eqref{14}, it is not hard to prove that   the sequence
$\{\mathfrak{D}^{\alpha}_{0+}  f_{n}\}$ is fundamental with respect to the norm of the space  $L_{2}(\Omega).$   Hence there  exists a limit   $\mathfrak{D}^{\alpha}_{0+}  f_{n}\stackrel{L_{2}}{\longrightarrow}\varphi.$ By virtue of the continuity property of the
 directional fractional integral  (see Theorem 2.1 \cite{firstab_lit:1kukushkin2018}),   we get
$ \mathfrak{I}_{0+}^{\alpha}\mathfrak{D}^{\alpha}_{0+}  f_{n}\stackrel{L_{2}}{\longrightarrow}\mathfrak{I}_{0+}^{\alpha}\varphi.$ On the other hand, the application  of   Theorem 2.4 \cite{firstab_lit:1kukushkin2018}, Theorem 2.3 \cite{firstab_lit:1kukushkin2018} yields   $\mathfrak{I}_{0+}^{\alpha}\mathfrak{D}^{\alpha}_{0+} f_{n}=  f_{n},\,n\in \mathbb{N}.$ Therefore  $  f_{n}\stackrel{L_{2}}{\longrightarrow}\mathfrak{I}_{0+}^{\alpha}\varphi.$ Since   $  f_{n}\stackrel{L_{2}}{\longrightarrow}  f,$  then by virtue of the uniqueness property of  limit,  we obtain  $  f=\mathfrak{I}_{0+}^{\alpha}\varphi.$   Applying Theorem 2.3 \cite{firstab_lit:1kukushkin2018}, we get $\mathfrak{D}^{\alpha}_{0+}  f=\varphi.$   In accordance with said above, we have  $\mathfrak{D}^{\alpha}_{0+}  f_{n}\stackrel{L_{2}}{\longrightarrow}\mathfrak{D}^{\alpha}_{0+}  f. $   Passing to the limit on the left and   right side of   inequality
\eqref{14}   we complete   the proof.
\end{proof}
\begin{lem}\label{L2}
 The   numerical range   $\Theta(L)$  belongs to a sector $\mathfrak{L}_{\gamma}(\theta),$
 where
 \begin{equation}\label{15}
\gamma=  \eta^{-2} -   a\frac{ \mathcal{I} }{2}\left(\frac{ \mathcal{I}}{2} \,\varepsilon^{2}+ \mathcal{A}\,\varepsilon\right)^{\!\!\!-1}  \!\!,\;\theta=\arctan  \left(   \frac{\mathcal{I}}{2a}\,\varepsilon     + \frac{\mathcal{A}}{a} \right)    ,\, \varepsilon\in(0,\infty).
\end{equation}
The parameter $\varepsilon$ can be chosen so that
\begin{equation}\label{16}
\gamma:=\left\{ \begin{aligned}
    \gamma<0,\;   \varepsilon\in(0,\xi) ,\\
 \gamma\geq 0,\;\varepsilon\in[\xi,\infty) \\
\end{aligned}
 \right.,\;\;\xi=\sqrt{\left(\frac{\mathcal{A}}{ \mathcal{I}}\right)^{2}+  a \eta^{2}   } -\frac{\mathcal{A}}{ \mathcal{I}}.
\end{equation}
In  the   case   $(\gamma=0),$   we have the following value of the semi-angle
 \begin{equation}\label{17}
 \theta  = \arctan\left\{ \sqrt{\left(\frac{\mathcal{A}}{ 2a}\right)^{2}+ \frac{   \mathcal{I}^{\,2}\eta^{2}}{4a }} +\frac{ \mathcal{A}}{2a}\right\}, \,\mathcal{I}=\mathcal{P}\mathcal{K}.
\end{equation}
\end{lem}
 \begin{proof}
We have the following estimate
$$
\left|{\rm Im} ( f,Lf    )_{L_{2}}\right|\leq   \left|\int\limits_{\Omega}
\left(a^{ij}  D_{i}u D_{j}v-a^{ij}  D_{i}v D_{j}u\right)dQ\right|
$$
$$
+\left|  ( u,\rho \,\mathfrak{D}^{\alpha}v )_{L_{2}} -( v,\rho\,\mathfrak{D}^{\alpha}u   )_{L_{2}}\right|= I_{1}+I_{2},
$$
where $f\in \mathrm{D}(L),\,f=u+iv.$
 Using the Cauchy-Schwarz inequalities  for    sums and   integrals, applying  the Jung  inequality, taking into account \eqref{3}, we get
\begin{equation}\label{18}
 \left|\int\limits_{\Omega}
 \!\!a^{ij}  D_{i}u D_{j}v dQ\right|\! \leq    \mathcal{A}\!\int\limits_{\Omega}
  |\nabla u|\,|\nabla v|\, dQ\leq \mathcal{A}
  \| u\|_{H_{0}^{1} }\| v\|_{H_{0}^{1}}\!\leq\!\frac{\mathcal{A}}{2}\left( \| u\|^{2}_{H_{0}^{1}}+\| v\|^{2}_{H_{0}^{1}}\right).
 \end{equation}
It follows   that
$
 I_{1}\!\leq  \! \mathcal{A}\|f\|^{2}_{H_{0}^{1}}.
$
Using  Lemma \ref{L1}, the  Jung  inequality, we obtain
\begin{equation}\label{19}
\left|( u,\rho\,\mathfrak{D}^{\alpha}v )_{L_{2} } \right|
  \leq \mathcal{I} \, \|u\|_{L_{2} }\| v\|_{H_{0}^{1}}\leq
 \frac{ \mathcal{I} }{2}\left (\frac{1}{\varepsilon}\|u\|^{2}_{L_{2} }+\varepsilon\| v\|^{2}_{H_{0}^{1}} \right);
 $$
 $$
\left|( v,\rho\,\mathfrak{D}^{\alpha}u )_{L_{2} } \right|\leq \frac{ \mathcal{I} }{2}\left (\frac{1}{\varepsilon}\|v\|^{2}_{L_{2} }+\varepsilon\| u\|^{2}_{H_{0}^{1}} \right) .
\end{equation}
 Hence
$$
 I_2   \leq \left|( u,\rho\,\mathfrak{D}^{\alpha}v )_{L_{2} } \right| +\left|( v,\rho\,\mathfrak{D}^{\alpha}u )_{L_{2} } \right|\leq \frac{ \mathcal{I} }{2}\left (\frac{1}{\varepsilon}\|f\|^{2}_{L_{2} }+\varepsilon\| f\|^{2}_{H_{0}^{1}} \right).
 $$
  Finally   we have the following estimate
$$
\left|{\rm Im} ( f,Lf    )_{L_{2} }\right|\leq \frac{\mathcal{I}  }{2}\,  \varepsilon^{-1} \|f\|^{2}_{L_{2} }+\left(\frac{\mathcal{I}  }{2} \,\varepsilon+ \mathcal{A}\right)\| f\|^{2}_{H_{0}^{1} }  .
 $$
Thus, taking  into account  \eqref{4},\eqref{5}, we obtain
$$
{\rm Re}( f,Lf    )_{L_{2} }-b \left|{\rm Im} ( f,Lf    )_{L_{2} }\right|\geq
 \left[a- b\left(\frac{ \mathcal{I} }{2} \,\varepsilon+ \mathcal{A}\right)\right] \|  f\|^{2}_{H_{0}^{1} }+
 \left(  \eta^{-2} -   \frac{ \mathcal{I}}{2}\, b \, \varepsilon^{-1}  \right)\|f\|^{2}_{L_{2} },\,b>0.
$$
Choosing a value of the parameter $b$ accordingly,  we get
\begin{equation}\label{20}
\left|{\rm Im} ( f,[L-\gamma] f    )_{L_{2} }\right|   \leq  \frac{1}{b(\varepsilon) } \,{\rm Re}( f,[L-\gamma]f    )_{L_{2} }  ,\;b(\varepsilon) = a\left(\frac{ \mathcal{I}}{2} \,\varepsilon+ \mathcal{A}\right)^{-1}\!\!\!\!,
$$
$$
 \gamma= \eta^{-2} -   \frac{ \mathcal{I}}{2}\, b(\varepsilon) \, \varepsilon^{-1}.
\end{equation}
Taking into account that  $\Theta(L)=\left\{\zeta:\, \zeta=\omega+\gamma,\,\omega\in \Theta(L-\gamma)\right\}, $ we conclude  that      $\Theta(L)\subset \mathfrak{L}_{\gamma}(\theta),$ where $\theta=\arctan  \{1 /b(\varepsilon)\} .$   Relation \eqref{15} is proved.
   Solving system of equations  \eqref{15} relative to $\varepsilon,$ we obtain  the positive root $\xi$ corresponding to  the value $\gamma=0.$ Also, we obtain     description \eqref{16}  for  coordinates of  a   sector vertex $\gamma.$
Let us consider in detail  the   case   $(\gamma=0).$   In this case $\varepsilon=\xi,$ hence $b (\xi)= a\left(  \xi\mathcal{I}/2+ \mathcal{A}\right)^{-1}.$ Using  the   expression for $\xi,$ by  easy calculation,  we get
\eqref{17}.
\end{proof}
\begin{rem}
Consider   the   sector $\mathfrak{L}_{0}(\theta),\,\theta=\arctan \{1/b(\xi)\}.$ Combining Theorem 4.3  \cite{firstab_lit:1kukushkin2018} with    Theorem 3.2 \cite[p.268]{firstab_lit:kato1980},
  we obtain     $0\,\cup(\mathbb{C}\setminus  \mathfrak{ L }_{0})\subset \mathrm{P}(L).$ Using
   Lemma \ref{L1}  it is not hard to prove that $\Theta(R_{L} )\subset \mathfrak{L}_{0}(\theta).$
\end{rem}

\section{Main results}
The following lemma is a central point of the method based on  properties of  the real component.

\begin{lem}\label{L3}
The following relation  holds
\begin{equation*}
\lambda_{i}\left[(R_{L})_{\mathfrak{R}}\right]\asymp
 i^{-2/n} .
\end{equation*}
 \end{lem}
\begin{proof}
  First  consider the next operators with   positive coefficients
$$
Y_{k}f=- \mu_{k}  \Delta f  +\sigma_{k}f,\;f\in \mathrm{D}(L),
$$
$$
  \mu_{k},\sigma_{k}={\rm const},\;k=0,1.
$$
It is well known fact that $Y^{\ast}_{k}=Y _{k},\,\Theta(Y _{k})\subset\{z\in \mathbb{R}:z>C\}.$
  Choose the coefficients  so that  $\mu_{0}=a, \sigma_{ 0 } = \eta^{-2},\mu_{1}=\mathcal{A}+ \mathcal{I}/2,\sigma_{1}= \mathcal{I}/2,$ then using \eqref{4},\eqref{5},\eqref{18},\eqref{19},  we obtain
\begin{equation}\label{21}
 (Y_{0}f,f)_{L_{2}} \leq(Hf,f)_{L_{2}}  \leq (Y_{1}f,f)_{L_{2}},\,f\in \mathrm{D}(L ),
\end{equation}
where $H:=L_{\mathfrak{R}}.$
 It is easy to prove that the energetic spaces   $\mathfrak{H}_{Y_{0}},\mathfrak{H}_{Y_{1}}$ coincide with the space $H_{0}^{1}(\Omega)$ as  sets of elements. Hence by virtue of inequality \eqref{21} it can easily be checked that the energetic space $\mathfrak{H}_{H}$ coincides with the space $H_{0}^{1}(\Omega)$ as a set of elements. Also it is clear that $(Y_{0}f,f)_{L_{2}}\geq C\|f\|_{H_{0}^{1}},\,f\in H_{0}^{1}(\Omega).$
 Hence due to the Rellich-Kondrashov theorem the considered above  energetic spaces   are compactly embedded    in  $L_{2}(\Omega).$    Applying    Theorem 1 \cite[p.225]{firstab_lit:mihlin1970}, we get
\begin{equation}\label{22}
 \lambda_{i}( Y_{0} )\leq \lambda_{i}(H)\leq\lambda_{i}( Y_{1}),\; i\in \mathbb{N}.
\end{equation}
Let us    use   asymptotic formula (3) \cite{firstab_lit:fedosov1964} for the counting function of the Laplace operator
\begin{equation} \label{23}
N_{-\Delta}(\lambda)=\frac{ {\rm mes \Omega}}{2^{n}\pi^{n/2}\Gamma(n/2+1)}\lambda^{n/2} +\mathcal{O} (\lambda^{ (n-1)/2}\ln\sqrt{\lambda}),\;\lambda\rightarrow \infty.
\end{equation}
Applying  formula (5.12) \cite{firstab_lit:2Agranovich2011}, we get
\begin{equation} \label{24}
 \lambda_{i}(-\Delta) =4\pi \left[\frac{\Gamma(n/2+1)}{ {\rm mes \Omega}}\right]^{2/n}i^{2/n} +\beta(i^{ 2/n}),\;\beta(i^{ 2/n})= o  (i^{ 2/n} ),\;i\rightarrow \infty.
\end{equation}
 It implies that
\begin{equation*}
 \lambda_{i}(Y_{k}) = \mathcal{E}_{k}i^{2/n} +  \phi_{k}  (i^{ 2/n} ),\;i\in \mathbb{N},\;
 $$
 $$
 \mathcal{E}_{k}=4\pi \mu_{k} \left[\frac{\Gamma(n/2+1)}{{\rm mes \Omega}}\right]^{2/n}\!\!\!,\;\;\; \phi_{k}  (i^{ 2/n} )=\beta(i^{ 2/n})\,\mu_{k}+\sigma_{k},\,\;\;k=0,1.
\end{equation*}
Applying   estimate \eqref{22}, we get
\begin{equation}\label{25}
 \mathcal{E}_{0}  i^{2/n}+ \phi_{\,0}  (i^{2/n})\leq \lambda_{i}(H) \leq   \mathcal{E}_{1} i^{2/n}+ \phi _{1} (i^{2/n})  ,\,i\in\mathbb{N}.
\end{equation}
Now we shall  show that the following relation  holds
\begin{equation}\label{26}
\lambda_{i}\left[(R_{L})_{\mathfrak{R}}\right]\asymp
 \lambda_{i}(R_{H}).
\end{equation}
 It is proved in Theorem 5.4 that  $H=ReL.$   By virtue of Theorem 4.3 \cite{firstab_lit:1kukushkin2018}, we know that  the operator   $L$ is m-sectorial, moreover due to Lemma \ref{L1} we have $\Theta(L)\subset \mathfrak{L}_{0}(\theta).$  Hence  using    Theorem 3.2 \cite[p.337]{firstab_lit:kato1980} we have the following representation
 \begin{equation}\label{27}
L=H^{\frac{1}{2}}(I+i B_{1}) H^{\frac{1}{2}},\;L^{\ast}=H^{\frac{1}{2}}(I+i B_{2}) H^{\frac{1}{2}},
\end{equation}
where   $B_{i}:=\left\{B_{i}:L_{2}(\Omega)\rightarrow L_{2}(\Omega),\,B^{\ast}_{i}=B_{i},\;\|B_{i}\|\leq \tan \theta \right\},\,i=1,2.$

Since the set of  linear operators  generates  ring (algebraic structure),  then we obtain
\begin{equation*}
  Hf =\frac{1}{2}\left[H^{\frac{1}{2}}(I+i B_{1})  +H^{\frac{1}{2}}(I+i B_{2})\right]H^{\frac{1}{2}}  = \frac{1}{2}\left\{H^{\frac{1}{2}}\left[(I+i B_{1})  + (I+i B_{2})\right]\right\}H^{\frac{1}{2}}
$$
$$
=  H f +
 \frac{i}{2} H^{\frac{1}{2}}\left(B_{1}+B_{2}\right)  H^{\frac{1}{2}}f  ,\;f\in \mathrm{D}(L).
\end{equation*}
 Consequently
\begin{equation}\label{28}
H^{\frac{1}{2}}\left(B_{1}+B_{2}\right) H^{\frac{1}{2}}f=0  ,\;f\in \mathrm{D}(L).
\end{equation}
Let us show that $B_{1}=-B_{2}.$
Since  the operator $H$ is m-accretive, then we have
$
(H+\zeta)^{-1}\in  \mathcal{B }(L_{2}),\,\mathrm{Re}\,\zeta>0.
$
Using this fact, we get
\begin{equation}\label{29}
{\rm Re}\left([H+\zeta]^{-1}Hf,f\right)_{L_{2}}={\rm Re}\left([H+\zeta]^{-1}[H+\zeta]  f,f\right)_{L_{2}}-{\rm Re}\left(\zeta\,[H+\zeta]^{-1}   f,f\right)_{L_{2}}
$$
$$
\geq \|f\|^{2}_{L_{2}}-|\zeta|\cdot\|(H+\zeta)^{-1}\|\cdot\|f\|^{2}_{L_{2}}=\|f\|^{2}_{L_{2}} \left(1-|\zeta|\cdot \|(H+\zeta)^{-1}\|\right),
$$
$$
\,\mathrm{Re}\,\zeta>0,\,f\in \mathrm{D}(L).
\end{equation}
Applying    inequality \eqref{8},   we obtain
$$
\| f\| _{L_{2}}\|(H+\zeta)^{-1}f\| _{L_{2}}\geq|(f,[H+\zeta]^{-1}f)|
 \geq ( {\rm Re} \zeta+C_{0} )\|(H+\zeta)^{-1}f\|^{2}_{L_{2}},\;f\in L_{2}(\Omega).
$$
It implies that
$$
\|(H+\zeta)^{-1}\|\leq ({\rm Re}\zeta +C_{0}  )^{-1},\;{\rm Re}\zeta>0.
$$
  Combining    this  estimate  and  \eqref{29}, we have
$$
{\rm Re}\left([H+\zeta]^{-1}Hf,f\right)_{L_{2}}\geq \|f\|_{L_{2}}^{2}\left(1-\frac{|\zeta|}{{\rm Re}\zeta +C_{0}}\right),\,{\rm Re}\zeta>0,\,f\in \mathrm{D}(L).
$$
Applying    formula (3.45) \cite[p.282]{firstab_lit:kato1980} and  taking into account that $H^{\frac{1}{2}}$ is selfadjoint, we get
\begin{equation}\label{30}
 \left(H^{\frac{1}{2} }f,f\right)_{L_{2}}=\frac{1}{\pi}\int\limits_{0}^{\infty}\zeta^{-1/2} \left([H+\zeta]^{-1}Hf,f\right)_{L_{2}}d\zeta
$$
$$
 \geq\|f\|^{2}_{L_{2}}\cdot \frac{C_{0}}{\pi }\int\limits_{0}^{\infty}\frac{\zeta^{-1/2}}{\zeta+C_{0}}  d\zeta=
 \sqrt{C_{0}}  \|f\|^{2}_{L_{2}},\,f \in \mathrm{D}(L) .
\end{equation}
Since in accordance with   Theorem 3.35 \cite[p.281]{firstab_lit:kato1980} the set      $ \mathrm{D} (H)$ is  a core of the operator $H^{ \frac{1}{2}},$    then we can extend \eqref{30} to
  \begin{equation}\label{31}
 \left(H^{\frac{1}{2}}f,f\right)_{\!\!L_{2}}\geq \sqrt{C_{0}} \|f\|^{2}_{L_{2}},\;f\in \mathrm{D}(H^{\frac{1}{2}}).
\end{equation}
Hence $ \mathrm{N} (H^{\frac{1}{2} })=0.$ Combining this fact and  \eqref{28}, we obtain
\begin{equation}\label{32}
 \left(B_{1}+B_{2}\right)  H^{ \frac{1}{2}}f=0  ,\;f\in \mathrm{D}(L).
\end{equation}
In accordance with    Theorem 3.35   \cite[p.281]{firstab_lit:kato1980} the operator $ H ^{ \frac{1}{2}} $ is m-accretive. Hence  combining  Theorem 3.2 \cite[p.268]{firstab_lit:kato1980}        with \eqref{31},
  we obtain   $ \mathrm{R} (H ^{ \frac{1}{2}})=L_{2}(\Omega).$ Taking into account that    $\mathrm{D}(H)=\mathrm{D}(L)$ is a core of the operator $H^{ \frac{1}{2}},$  we   conclude that
$ \mathrm{R}(\check{H} ^{\frac{1}{2} })$ is dense in $L_{2}(\Omega),$ where $\check{H} ^{ \frac{1}{2} }$ is a restriction of  the operator $H ^{ \frac{1}{2}}$ to $L_{2}(\Omega).$ Finally, by virtue of \eqref{32}, we have that the sum   $B_{1}+B_{2}$ equals zero on the dense subset  of $L_{2}(\Omega).$ Since these operators are   defined on $L_{2}(\Omega)$ and  bounded,  then   $B_{1}=-B_{2}.$   Further,   we    use the denotation  $B_{1}:=B.$
 Using the  properties  of the operator $B,$    we get
$\|(I\pm iB)f\|_{L_{2} }\|f\|_{L_{2} }\geq\mathrm{Re }\left([I\pm iB]f,f\right)_{L_{2}} =\|f\|^{2}_{L_{2}},\,f\in L_{2}(\Omega).$ Hence
\begin{equation}\label{33}
\|(I\pm iB)f\|_{L_{2} } \geq \|f\|_{L_{2}},\,f\in L_{2}(\Omega).
\end{equation}
 It implies that the operators  $I\pm iB$ are invertible.
Using  the facts    proven above: $  \mathrm{R} (H^{ \frac{1}{2}})=L_{2}(\Omega),\, \mathrm{N} (H^{\frac{1}{2}})=0$,     we   conclude that there exists the operator $H^{-\frac{1}{2}}$ defined on $L_{2}(\Omega).$
 Taking into account  the   reasoning given above, using    representation   \eqref{27}, we get the following
\begin{equation}\label{34}
R_{L}=H^{-\frac{1 }{2}}(I+iB )^{-1} H^{- \frac{1}{2}},\;R_{L^{\ast}}=H^{-\frac{1 }{2}}(I-iB )^{-1} H^{- \frac{1}{2}}.
\end{equation}
Since we have  $R^{\ast}_{ L}=R^{\,}_{L^{\ast}},$ then
\begin{equation}\label{35}
V=\frac{1}{2}\left(R_{ L}+R_{L^{\ast}}\right),\,V:=  \left(R_{\tilde{W}}\right)_{\mathfrak{R}}.
\end{equation}
Combining  \eqref{34},\eqref{35},  we get
\begin{equation}\label{36}
V=\frac{1}{2}\,H^{-\frac{1 }{2}}\left[(I+iB )^{-1}+(I-iB )^{-1} \right]H^{- \frac{1}{2}}.
\end{equation}
 Using the obvious identity
$
(I+B^{2})=(I+iB ) (I-iB )= (I-iB )(I+iB ),
$
 by  direct calculation  we   get
\begin{equation}\label{37}
(I+iB )^{-1}+(I-iB )^{-1}=(I+B^{2})^{-1}.
\end{equation}
Combining \eqref{36},\eqref{37}, we obtain
\begin{equation}\label{38}
V=\frac{1}{2}\,H^{-\frac{1 }{2}}  (I+B^{2} )^{-1}  H^{- \frac{1}{2}}.
\end{equation}
Let us evaluate the form $\left(V  f,f\right)_{L_{2}}.$   Note that   by virtue of   Theorem \cite[p.174]{firstab_lit:Krasnoselskii M.A.}    there exists a unique  square root of the operator $ R_{H} $ the  selfadjoint operator $ \hat{R} $ such that
  $\hat{R} \hat{R}  =R_{H}.$     Using the    decomposition $H=H^{\frac{1}{2}}H^{\frac{1}{2}},$   we get   $H^{-\frac{1}{2}}H^{-\frac{1}{2}}H=I.$ Therefore
$R_{H}\subset H^{-\frac{1}{2}}H^{-\frac{1}{2}}.$  On the over hand    $  \mathrm{D} (R_{H})=L_{2}(\Omega),$ hence  $R_{H}=H^{-\frac{1}{2}}H^{-\frac{1}{2}}.$  Using  the  uniqueness property  of    square root  we   obtain       $H^{-\frac{1}{2}}= \hat{R}.$      Note that due to the obvious inequality   $ \|Sf\|_{L_{2}}  \geq\|f\|_{L_{2}},\,f\in L_{2}(\Omega) ,\,S:=I+B^{2}$  the  operator $S^{-1}$ is bounded on the set $ \mathrm{R} (S).$ Taking into account the reasoning given above, we get
 $$
\left(V  f,f\right)_{L_{2}}=\left(H^{-\frac{1 }{2}} S^{-1}     H^{- \frac{1}{2}}   f,f\right)_{L_{2}}=
\left( S^{-1}     H^{- \frac{1}{2}}   f,H^{-\frac{1 }{2}}f\right)_{L_{2}}
$$
$$
\leq \|S^{-1}     H^{- \frac{1}{2}}   f \| _{L_{2}}  \|H^{- \frac{1}{2}}   f \|_{L_{2}}\leq   \|S^{- 1}\|  \cdot  \|   H^{- \frac{1}{2}}   f \|^{2}_{L_{2}}
=
  \|S^{- 1}       \|\cdot\left(R_{H }  f,f\right)_{L_{2}},\;f\in L_{2}(\Omega).
$$
On the other hand,  it is easy to see that  $(S^{-1}f,f)_{L_{2}}\geq \|S^{-1}f\|^{2}_{L_{2}},\,f\in \mathrm{ R} (S).$ At the same time   it is obvious that   $S$ is bounded and we have   $\|S^{-1}f\|_{L_{2}}\geq \|S\|^{-1} \|f\|_{L_{2}},\,f\in  \mathrm{R} (S).$ Using   these estimates, we get
$$
\left(V f,f\right)_{L_{2}}=\left( S^{-1}     H^{- \frac{1}{2}}   f,H^{-\frac{1 }{2}}f\right)_{L_{2}}\geq  \|S^{-1}     H^{- \frac{1}{2}}   f \|^{2}_{L_{2}}
$$
$$
\geq  \| S   \|^{-2} \cdot  \|   H^{- \frac{1}{2}}   f \|^{2}_{L_{2}}= \| S   \|^{-2}  \cdot\left(R_{H }  f,f\right)_{L_{2}},\;f\in L_{2}(\Omega).
$$
 Taking into account that  the operators $V $ and $R_{H}$ are compact,  using   Lemma 1.1 \cite[p.45]{firstab_lit:1Gohberg1965}, we obtain
 \eqref{26}. Combining \eqref{25},\eqref{26} we obtain the desired result.
\end{proof}

The following theorem gives us a description in terms of the Schatten–von-Neumann classes.

\begin{teo}\label{T1} We have the following  sufficient and necessary conditions
\begin{equation*}
R_{L}\in  \mathfrak{S}_{p},\,p= \left\{ \begin{aligned}
\!l,\,l>n,\,n\geq2,\\
   1,\,n=1    \\
\end{aligned}
 \right.\;,
\end{equation*}
$$
 R_{ L}\in\mathfrak{S}_{p}\;  \Rightarrow \;  n< 2p,\;1\leq p<\infty.
$$
 \end{teo}
\begin{proof}
  Consider the case $(n\geq2).$ Since we already know that $R_{  L   }^{*}=R_{L^{\ast} }^{\!},$ then it can easily be checked that
the operator $F\!:= R_{  L   }^{*}  R_{L  }^{\!} $   is a selfadjoint positive compact operator. Due to the well-known fact   \cite[p.174]{firstab_lit:Krasnoselskii M.A.} there exists a unique  square root of the operator $F,$ the selfadjoint  positive operator $F^{\frac{1}{2}}$ such that $F=F^{\frac{1}{2}}F^{\frac{1}{2}}.$ By virtue  of Theorem 9.2 \cite[p.178]{firstab_lit:Krasnoselskii M.A.} the operator $F^{\frac{1}{2}}$ is compact.
Since $ \mathrm{N} (F) =0,$  it follows that $ \mathrm{N} (F^{\frac{1 }{2}})=0.$ Hence      applying   Theorem \cite[p.189]{firstab_lit: Ahiezer1966}, we get that the operator
$F^{\frac{1 }{2}}$ has an infinite  set of the eigenvalues.   Using  condition \eqref{6}, we get
$$
{\rm Re}(R_{ L}f,f)_{L_{2} }\geq C_{1}\|R_{ L}f\|^{2}_{L_{2}},\;f\in L_{2}(\Omega).
$$
Hence
$$
(F f,f)_{L_{2}}=\|R_{ L}f\|^{2}_{L_{2}}\leq C^{-1}_{1}{\rm Re}(R_{ L}f,f)_{L_{2}}= C^{-1}_{1}(V f,f)_{L_{2}},\,V:=  \left(R_{ L}\right)_{\mathfrak{R}}.
$$
Since we already know that the operators $F,V$ are compact, then  using  Lemma 1.1 \cite[p.45]{firstab_lit:1Gohberg1965}, Lemma \ref{L3},  we get
\begin{equation}\label{39}
 \lambda_{i} (F )\leq C^{-1}_{1} \,\lambda_{i}(V)\leq  C  i^{-2/n },\;i\in \mathbb{N}.
\end{equation}
 Recall  that by  definition, we have     $s_{i}(R_{ L})= \lambda_{ i }( F^{\frac{1 }{2}}  ).$ Note that the operators $F^{\frac{1}{2}},F$ have the  same eigenvectors.
This fact can be easily proved if we note   the obvious  relation
$
Ff_{i}=|\lambda _{i} (F^{\frac{1}{2}})|^{2} f_{i},\, i\in \mathbb{N}
$
and
  representation for a square root of a  selfadjoint positive compact operator
$$
F^{\frac{1 }{2}}f=\sum\limits_{i=1}^{\infty}\sqrt{\lambda _{ i }(F)}  \left(f,\varphi_{i}\right)_{\!L_{2}} \! \varphi_{i}, \,f\in L_{2}(\Omega),
$$
where  $f_{i}\,, \varphi_{i}$ are    eigenvectors  of the  operators $F^{\frac{1}{2}},F$ respectively  (see (10.25) \cite[p.201]{firstab_lit:Krasnoselskii M.A.}).  Hence
$ \lambda_{ i }( F^{\frac{1 }{2}}  ) =
\sqrt{\lambda _{ i }( F  )} ,\,i\in \mathbb{N}.$
Combining this fact with  \eqref{39}, we get
$$
\sum\limits_{i=1}^{\infty}s^{p}_{i}(R_{ \tilde{W}})=
\sum\limits_{i=1}^{\infty}\lambda_{i}^{\frac{p}{2} }( F   )\leq C \sum\limits_{i=1}^{\infty} i^{-\frac{  p }{n}}.
$$
This completes the proof for the case $(n\geq2).$

Consider the case $(n=1).$ Since  $V$ is positive     and   bounded, then   by virtue  of Lemma 8.1 \cite[p.126]{firstab_lit:1Gohberg1965}  we have that for any  orthonormal basis $\{\psi_{i}\}_{1}^{\infty}\subset L_{2}(\Omega)$ the next equalities  hold
$$
\sum\limits_{i=1}^{\infty}{\rm Re}(R_{L}\psi_{i},\psi_{i})_{L_{2}}=\sum\limits_{i=1}^{\infty}  (V \psi_{i},\psi_{i})_{L_{2}}=\sum\limits_{i=1}^{\infty}  (V\, \varphi_{i},\varphi_{i})_{L_{2}} ,
$$
where $\{\varphi_{i}\}_{1}^{\infty}$ is an orthonormal  basis of the  eigenvectors of the operator $V.$
Due to  Lemma   \ref{L3}, we get
$$
  \sum\limits_{i=1}^{\infty}  (V \varphi_{i},\varphi_{i})_{L_{2}} =\sum\limits_{i=1}^{\infty}
 s_{i}(V) \leq C  \sum\limits_{i=1}^{\infty}  i^{-\frac{2}{n}  }  .
$$
By virtue of  Lemma \ref{L2}, we have $ |{\rm Im} (R_{L}\psi_{i},\psi_{i})_{L_{2}}|\leq b^{-1} (\xi)\, {\rm Re}(R_{L}\psi_{i},\psi_{i})_{L_{2}}.$ Hence  the  following series converges, i.e.
$$
\sum\limits_{i=1}^{\infty} (R_{L}\psi_{i},\psi_{i})_{L_{2}}<\infty.
$$
Hence, by definition \cite[p.125]{firstab_lit:1Gohberg1965} the operator $R_{L}$ has  a finite matrix trace.
 Using    Theorem 8.1 \cite[p.127]{firstab_lit:1Gohberg1965}, we get   $R_{ L}\in \mathfrak{S}_{1}.$ This completes the proof for the case $(n=1).$

Now, assume that  $R_{L}\in \mathfrak{S}_{p},\,1\leq p<\infty.$   Recall that $V$ is  compact  and  let us show that the   operator $V$ has a complete orthonormal  system of the  eigenvectors.   Using   formula \eqref{39},   we get
$$
V^{\!^{-1}}=2H^ \frac{1 }{2}  (I+B^{2})      H^{ \frac{1}{2}},\;  \mathrm{D}  (V^{\!^{-1}})=  \mathrm{R} (V).
$$
It can easily be checked that   $ \mathrm{D} (V^{\!^{-1}})\subset  \mathrm{D} (H).$
Using this fact, we get
\begin{equation}\label{40}
 (V^{^{-1}}\!\!f,f)_{L_{2}}  =2(  S      H^{ \frac{1}{2}}f,H^ \frac{1 }{2}f)_{L_{2}}\geq 2\|H^ \frac{1 }{2} f\|^{2}_{L_{2}}=2(Hf,f)_{L_{2}}  ,\,f\in  \mathrm{D} (V^{^{-1}}),
\end{equation}
where $S=I+B^{2}.$
 Since  $V$ is selfadjoint, then   due to   Theorem 3 \cite[p.136]{firstab_lit: Ahiezer1966}    $V^{^{-1}}$ is selfadjoint. Combining  \eqref{8} and \eqref{40}, we obtain  $\Theta(V^{^{-1}}) \subset \{ z\in    \mathbb{R}  :  z > C\}. $
 Since  the operator $H$ has a discrete spectrum (see Theorem 5.3 \cite{firstab_lit:1kukushkin2018}), then any set  bounded   with respect to the      norm $\mathfrak{H}_{H}$ is a compact set   with respect to the norm    $L_{2}(\Omega)$ (see Theorem 4 \cite[p.220]{firstab_lit:mihlin1970}). Combining this fact with \eqref{40}, Theorem 3 \cite[p.216]{firstab_lit:mihlin1970}, we get
  that the operator $V^{^{-1}}$ has a discrete spectrum, i.e. it has  an infinite set of  the eigenvalues
$\lambda_{1}\leq\lambda_{2}\leq...\leq\lambda_{i}\leq..., \, \lambda_{i} \rightarrow \infty  ,\,i\rightarrow \infty$ and a  complete orthonormal system of the eigenvectors.
 Now  note that the operators $V,\,V^{^{-1}}$ have the same eigenvectors.  Therefore   the operator   $V$ has a complete orthonormal  system of the eigenvectors. Recall  that any  complete orthonormal system  is a  basis in  separable Hilbert space. Hence the system of the  eigenvectors of the operator $V$ is
  a basis in the space $L_{2}(\Omega).$
 Let  $\{\varphi_{i}\}_{1}^{\infty}$ be a complete orthonormal set of the  eigenvectors of the  operator  $V.$    Using    inequalities  (7.9) \cite[p.123]{firstab_lit:1Gohberg1965}, Lemma  \ref{L3},  we get
$$
\sum\limits_{i=1}^{\infty}|s_{i} (R_{L} )|^{p}\geq\sum\limits_{i=1}^{\infty}| (R_{L}\varphi_{i},\varphi_{i})_{L_{2}}|^{p}\geq\sum\limits_{i=1}^{\infty}|{\rm Re}(R_{L}\varphi_{i},\varphi_{i})_{L_{2}}|^{p}=
$$
$$
=\sum\limits_{i=1}^{\infty}| (V \varphi_{i},\varphi_{i})_{L_{2}}|^{p}=\sum\limits_{i=1}^{\infty}
|\lambda_{i}(V)|^{p}\geq C   \sum\limits_{i=1}^{\infty}  i^{- \frac{2 p }{n}}  .
$$
We claim   that $n<2/p.$  Assuming the converse in the last inequality,  we come to contradiction with the imposed      condition  $R_{ L}\in\mathfrak{S}_{p}.$
 \end{proof}

     The following theorem establishes a sufficient condition for  completeness   of the root functions system  of the operator $R_{L}.$
   This result  is formulated for the multidimensional case, however it is    remarkable that  the obtained sufficient condition   gives us an opportunity to claim    that  a     root functions system is always  complete in the cases of  the dimensions  $n= 1,2.$

\begin{teo}\label{T2}
If   $\theta< \pi/n,$  then   the     root  functions system of the operator $R_{L}$ is complete,   where $\theta$ is the semi-angle of the  sector $\mathfrak{L}_{0}(\theta).$
 \end{teo}
\begin{proof}

We mentioned above that    $ \Theta (R_{L})\subset \mathfrak{ L }_{0}(\theta).$  Note that the map
     $z: \mathbb{C}\rightarrow \mathbb{C},\;  z=1/\zeta,$ takes each    eigenvalue  of the operator $R_{L}$  to the eigenvalue  of the operator $L.$ It is also clear    that   $z:\mathfrak{ L }_{0}(\theta)\rightarrow \mathfrak{ L }_{0}(\theta).$
  Using   the definition  \cite[p.302]{firstab_lit:1Gohberg1965} let us consider the following set
\begin{equation*}
 \widetilde{\mathcal{W}}_ {R_{L}} :=\left\{z:\,z=t\, \xi,\,\xi\in \widetilde{\Theta}(R_{L}),\, 0\leq t<\infty\right\}.
\end{equation*}
It is easy to see  that $\widetilde{\mathcal{W}}_ {R_{L}}$ coincides with a closed sector of the complex plane  with the vertex situated  at the point zero. Let us denote  by
 $\vartheta(R_{L})$ the angle of this sector. It is obvious that $\widetilde{\mathcal{W}}_ {R_{L}}\subset \mathfrak{ L }_{0}(\theta).$ Therefore    $ 0 \leq\vartheta(R_{L}) \leq  2\theta.$ Let us prove that     $0 <\vartheta(R_{L}),$ i.e.    strict inequality holds.   If we assume  that $\vartheta(R_{L})=0,$ then we get  $ e^{-i \mathrm{arg} z} =\varsigma,\,\forall z\in\widetilde{\mathcal{W}}_ {R_{L}}\setminus 0, $
 where $\varsigma$  is a constant independent on $z.$ In consequence of this fact   we have $\mathrm{Im}\, \Theta (\varsigma R_{L})=0.$ Hence the operator $\varsigma R_{L}$ is symmetric (see Problem  3.9 \cite[p.269]{firstab_lit:kato1980}) and by virtue of the fact   $  \mathrm{D}  (\varsigma R_{L})=L_{2}(\Omega)$ one is selfadjoint. On the other hand, taking into account   the equivalence   $R^{\ast}_{ L}=R^{\,}_{ L^{\ast}},$     we have $(\varsigma R_{L}f,g)_{L_{2}}=( f,\bar{\varsigma} R_{L^{\ast}}g)_{L_{2}},\,  f,g\in L_{2}(\Omega).$ Hence
 $\varsigma R_{L}=\bar{\varsigma} R_{L^{\ast}}.$ In the particular case, we have $\forall f\in L_{2}(\Omega),\,\mathrm{Im}f=0:\, \mathrm{Re}\,\varsigma\,R_{L}f= \mathrm{Re}\,\varsigma\,R^{\,}_{ L^{\ast}}f,\,\mathrm{Im}\,\varsigma\,R_{L}f= -\mathrm{Im}\,\varsigma\,R^{\,}_{ L^{\ast}}f  .  $ It implies that  $ \mathrm{N} (R_{L})\neq 0.$  This contradition concludes the proof of the fact $ \vartheta(R_{L})>0.$
    Let us use   Theorem 6.2 \cite[p.305]{firstab_lit:1Gohberg1965}  according to  which  we have the following. If the next two conditions  (a) and (b) are fulfilled, then    the set of   root functions  of the operator $R_{L}$ is complete.  \\

\noindent  a) $\vartheta(R_{L})=  \pi/ d ,$ where $d>1,$\\

\noindent b) for some $\beta,$ the  operator $B:=\left(e^{ i \beta}R_{L} \right)_{\mathfrak{I}}:\;s_{i}(B)=o(i^{-1/d }),\,i\rightarrow \infty .$\\

\noindent Let us show that conditions (a) and (b) are fulfilled. Note that due to Lemma \ref{L2} we have  $0\leq\theta<\pi/2.$
Hence $ 0<\vartheta(R_{L})<\pi.$ It implies that  there exists  $1<d<\infty$ such that
$\vartheta(R_{L})=  \pi/ d.$  Thus   condition  (a) is fulfilled.
  Let us choose a certain value  $\beta=\pi/2$ in   condition (b) and notice  that $\left(e^{ i \pi/2}R_{L} \right)_{\mathfrak{I}}=\left( R_{L} \right)_{\mathfrak{R}}. $ Since the operator $V:=\left( R_{L} \right)_{\mathfrak{R}}$ is  selfadjoint, it follows that $s_{i}(V)=\lambda_{i}(V),\,i\in \mathbb{N}.$
  In consequence of Lemma  \ref{L3}, we have
\begin{equation*}
 s_{i}(V)\, i^{1/d}    \leq C\cdot i^{1/d-2/n} ,\,i\in \mathbb{N}.
\end{equation*}
 Hence,     it is sufficient  to show that $d>n/2$ to achieve  condition  (b).
By virtue of the conditions $ \vartheta(R_{L}) \leq 2 \theta,\,\theta< \pi/n,$ we have $d=\pi/\vartheta(R_{L})\geq \pi/2\theta>n/2.$ Hence we obtain  $s_{i}(V)=o(i^{-1/d }).$
 Since   both    conditions (a),(b) are fulfilled, then using   Theorem 6.2 \cite[p.305]{firstab_lit:1Gohberg1965} we complete the proof.
\end{proof}
\begin{rem} Assume that the root functions system is complete; then using the definition, it is easy to prove that
  the set of the eigenvalues  is infinite. Thus  in the cases (n=1,2) we have an infinite set of the eigenvalues.
\end{rem}

The theorems proven above is devoted  to   description of  $s$-numbers    of the operator $R_{L},$    but    questions related with    the  eigenvalues asymptotic    are still relevant in our work. The following theorem establishes an asymptotic formula.
\begin{teo}\label{T3}
   The following inequality   holds
 \begin{equation}\label{41}
\sum\limits_{i=1}^{k}|\lambda_{i}(R_{L})|^{p}\leq
\sec ^{p} \theta \,\left\|S^{- 1}      \right\| \sum\limits_{i=1}^{k } \,\gamma^{p}_{0}(i)\, i^{- 2p/n  },
$$
$$
 (k=1,2,...,\nu(R_{L})),\,1\leq p<\infty,
\end{equation}
where $\theta=\arctan  \{  1/b(\xi) \} ,\, \gamma_{ 0 } (i)  =[\mathcal{E} _{0}+  \phi_{0}(i^{2/n})/i^{2/n}]^{-1}.$
Moreover if $\nu(R_{L})=\infty,$  then the following asymptotic formula holds
$$
|\lambda_{i}(R_{L})|=  o \left(i^{-2/n+\varepsilon}   \right),\,i\rightarrow \infty,\,\forall\varepsilon>0.
$$
\end{teo}

\begin{proof}
In accordance with  Theorem 6.1 \cite[p.81]{firstab_lit:1Gohberg1965}, we have the following inequality for any   bounded linear  operator $A$ with the compact imaginary component
 \begin{equation}\label{42}
\sum\limits_{m=1}^{k}|{\rm Im}\,\lambda_{m}(A)|^{p}\leq \sum\limits_{m=1}^{k}
|s_{m} (A_{\mathfrak{I} } )|^{p},\;\left(k=1,\,2,...,\,\nu_{\,\mathfrak{I}}(A)\right),\,1\leq p<\infty,
\end{equation}
where   $ \nu_{\,\mathfrak{I}}(A) \leq \infty $ is a sum of all  algebraic multiplicities corresponding to the  not real eigenvalues   of the operator $A$ (see   \cite[p.79]{firstab_lit:1Gohberg1965}).     We can also  verify the following relation by direct calculation
 \begin{equation}\label{43}
(iA)_{\mathfrak{I}}=A_{\mathfrak{R}},\; \mathrm{Im }\, \lambda_{m} (i\,A)= \mathrm{Re}\,  \lambda_{m}(A),\;m\in \mathbb{N}.
\end{equation}
By virtue of  the  fact $ \Theta (R_{L})\subset \mathfrak{ L }_{0}(\theta),$ we have ${\rm Re}\,\lambda_{m}(R_{L})>0,\,m=1,2,...,\nu\left(  R_{L} \right) .$ Combining this fact with  \eqref{43}, we obtain $\nu_{\mathfrak{I}}(i R_{L} )=\nu\left(R_{L}\right).$
Taking into account the last identity and combining \eqref{42},\eqref{43},   we obtain
 \begin{equation}\label{44}
\sum\limits_{m=1}^{k}|{\rm Re}\,\lambda_{m}(R_{L})|^{p}\leq \sum\limits_{m=1}^{k}
|s_{m} (V )|^{p} ,\; \left(k=1,\,2,...\,,\nu  (R_{L})\right).
\end{equation}
Due to   Lemma  \ref{L2}, we have
$$
|{\rm Im}\,\lambda_{m}(R_{L})|\leq  \tan \theta \,{\rm Re}\,\lambda_{m}\,(R_{L}),\,m\in \mathbb{N},\,.
$$
Hence, we get
 \begin{equation}\label{45}
| \lambda_{m}(R_{L})|= \sqrt{|{\rm Im}\,\lambda_{m}(R_{L})|^{2}+|{\rm Re}\,\lambda_{m}(R_{L})|^{2}}\leq
$$
$$
 \leq|{\rm Re}\,\lambda_{m}(R_{L})|\cdot \sqrt{\tan^{2}\theta+1}
=\sec \theta\, |{\rm Re}\,\lambda_{m}(R_{L})|,\,m\in \mathbb{N}.
\end{equation}
Combining \eqref{44},\eqref{45}, we obtain
 \begin{equation}\label{46}
\sum\limits_{m=1}^{k}| \lambda_{m}(R_{L})|^{p}\leq \sec^{p}\!\theta\,\sum\limits_{m=1}^{k}
|s_{m} (V )|^{p} ,\; \left(k=1,\,2,...\,,\nu  (R_{L})\right).
\end{equation}
Applying    Lemma \ref{L3}, we complete     the proof of inequality \eqref{41}.

Assume that $\nu(R_{L})=\infty.$  Note   that  the series on  the right side of \eqref{41} converges  in the case $(p>n/2).$  It implies that
 \begin{equation}\label{47}
|\lambda_{i}(R_{L})| i^{1/p}\rightarrow 0,\;i\rightarrow\infty.
\end{equation}
 We can choose a value of $p$ so that for     arbitrary small $\varepsilon >0,$ we have  $   1/p < 2/n ,\,  1/p>2/n-\varepsilon. $ Hence
$$
|\lambda_{i}(R_{L})| i^{ 2/n-\varepsilon}<|\lambda_{i}(R_{L})| i^{ 1/p} ,\,i\in \mathbb{N}.
$$
Applying relation \eqref{47}, we complete the proof of the asymptotic formula.
\end{proof}

\subsection*{Acknowledgments}

  The author warmly   thanks   Suslina  T. A.  for  valuable   comments made during the  report devoted  to the results of this work  and took place 06.12.2017    at   the seminar of Department of mathematical physics St. Petersburg state University, Saint  Petersburg branch of  V.A. Steklov Mathematical Institute of the Russian Academy of science, Russia, Saint  Petersburg.\\
Also enormous gratefulness is expressed to   Shkalikov A. A.  for a  number of valuable recommendations and remarks.

\end{document}